\documentclass[11pt]{amsart}
\title{Special divisors on marked chains of cycles}
\author[N. Pflueger]{Nathan Pflueger}\address{Department of Mathematics, Brown University, Box
1917, Providence, RI 02912}\email{pflueger@math.brown.edu}
\date{\today}

\usepackage{url}
\usepackage{color}
\usepackage{xifthen}
\usepackage{tikz}
\usetikzlibrary{decorations.pathreplacing,snakes}
\usetikzlibrary{patterns}
\usepackage{mathabx}
\usepackage{youngtab}
\usepackage{amssymb}
\usepackage{array}

\usepackage{setspace}

\theoremstyle{plain}
\newtheorem{thm}{Theorem}[section]
\newtheorem{lemma}[thm]{Lemma}
\newtheorem{prop}[thm]{Proposition}

\newtheorem{cor}[thm]{Corollary}

\theoremstyle{definition}
\newtheorem{eg}[thm]{Example}

\newtheorem{defn}[thm]{Definition}
\newtheorem{conv}[thm]{Convention}
\newtheorem{notn}[thm]{Notation}
\newtheorem{sit}[thm]{Situation}
\theoremstyle{remark}

\newtheorem{rem}[thm]{Remark}
\newtheorem{qu}[thm]{Question}

\newcommand{\ZZ}{\textbf{Z}}

\newcommand{\TT}{\textbf{T}}
\newcommand{\RR}{\textbf{R}}

\newcommand{\iou}[1][]{
    \ifthenelse{\equal{#1}{}}{{\color{blue}\{IOU\}}}
    {{\color{blue}\{IOU: #1\}}}
}
\newcommand{\la}[1]{\langle #1 \rangle}
\newcommand{\wl}{W^\lambda}
\newcommand{\wlgw}{W^\lambda(\Gamma,w_g)}
\newcommand{\um}{\underline{m}}

\DeclareMathOperator{\Pic}{Pic}
\DeclareMathOperator{\Div}{Div}
\DeclareMathOperator{\disp}{disp}

\DeclareMathOperator{\Trop}{Trop}
\DeclareMathOperator{\hook}{hook}

\begin{document}
\maketitle

\begin{abstract}
We completely describe all Brill-Noether loci on metric graphs consisting of a chain of $g$ cycles with arbitrary edge lengths, generalizing work of Cools, Draisma, Payne, and Robeva. The structure of these loci is determined by displacement tableaux on rectangular partitions, which we define. More generally, we fix a marked point on the rightmost cycle, and completely analyze the loci of divisor classes with specified ramification at the marked point, classifying them using displacement tableaux. Our results give a tropical proof of the generalized Brill-Noether theorem for general marked curves, and serve as a foundation for the analysis of general algebraic curves of fixed gonality.
\end{abstract}

\section{Introduction}

The tropical proof of the Brill-Noether theorem \cite{cdpr} gave a parameterization of the sets $W^r_d(\Gamma)$ of special divisors on a metric graph $\Gamma$ composed of a chain of cycles with generic edge lengths. The purpose of this paper to to generalize this analysis, in two ways. We use these generalizations to give a tropical proof of the generalized Brill-Noether theorem (for curves with a marked point), and also as a foundation for our paper \cite{pfl2}, which gives applications to the geometry of general curves of fixed gonality.

The first generalization is that we consider \textit{arbitrary} chains of cycles, with any edge lengths. We show that the behavior of the loci of special divisor classes on $\Gamma$ depends only on a sequence of numbers $\um = (m_2,m_3,\cdots,m_g)$, easily computed from the edge lengths, called the \emph{torsion profile} of $\Gamma$ (Definition \ref{def:tp}). The torsion profile determines whether the chain $\Gamma$ is Brill-Noether general in the sense of \cite{cdpr}. The metric graph $\Gamma$ is called \emph{Brill-Noether general} if $\dim W^r_d(\Gamma) = g-(r+1)(g-d+r)$ whenever this number is nonnegative, and $W^r_d(\Gamma)$ is empty otherwise. The original genericity condition of \cite{cdpr} is equivalent to saying that $m_i = 0$ or $m_i > 2g-2$.

\begin{thm} \label{thm_genericity1}
A chain of cycles $\Gamma$ is Brill-Noether general if and only if for each $i \in \{2,3,\cdots,g-1\}$, either $m_i = 0$ or $$m_i > \min( i, g+1-i).$$
\end{thm}

The second direction in which we generalize \cite{cdpr} is by considering special divisors with prescribed ramification at a marked point $w \in \Gamma$, which we always take to be on the rightmost cycle of the chain. In algebraic geometry, imposing ramification conditions amounts to considering only those linear series that correspond to maps to projective space in which the marked point is inflected in a particular way (see Definition \ref{def:wrad}). Although this generalization is interesting in its own right, it is in fact a necessary ingredient in our arguments even in the ordinary case. We will define, for any metric graph $\Gamma$ with a marked point $w$, generalized Brill-Noether loci $W^\lambda(\Gamma,w) \subseteq \Pic^0(\Gamma)$, where $\lambda$ is a partition. 

Here by a \emph{partition} we mean a finite, non-increasing sequence of nonnegative integers, where two such sequences are considered the same if one is obtained by adding a sequence of $0$s to the end of the other. We will also identify partitions with their Young diagrams (in French notation), which we in turn regard as subsets of $\ZZ^2_{>0}$, according to the following convention.

\begin{conv} \label{conv_youngdiagram}
Any partition $\lambda = (\lambda_0,\lambda_1,\cdots,\lambda_r)$ (where $\lambda_0 \geq \lambda_1 \geq \cdots \geq \lambda_r$) will be identified with the subset $$\{(x,y) \in \ZZ^2_{>0}:\ 1 \leq x \leq \lambda_{y-1},\ 1 \leq y \leq r+1 \},$$ which we will refer to as the Young diagram of $\lambda$. We denote the number of elements in this set (i.e. the sum of the elements $\lambda_i$) by $|\lambda|$.
\end{conv}

The link between the generalized Brill-Noether loci $W^\lambda(\Gamma,w)$ and the ordinary Brill-Noether loci is that $W^r_d(\Gamma)$ is isomorphic to $W^\lambda(\Gamma,w)$, where $\lambda$ is a \textit{rectangular} partition of height $r+1$ and width $g-d+r$, and $w$ is any marked point. The isomorphism $W^\lambda(\Gamma,w) \rightarrow W^r_d(\Gamma)$ is given by $[D] \mapsto [D + d \cdot w]$.

\begin{center}$\begin{array}{m{7cm} m{5cm}}
\begin{tikzpicture}[scale=0.5]
\draw (-0.5,2) node[left] {$\lambda = $};
\draw (0,0) rectangle (6,4);
\foreach \i in {1,2,...,5} \draw (\i,0) -- (\i,4);
\foreach \j in {1,2,3} \draw (0,\j) -- (6,\j);
\draw[snake=brace,thick] (6,-0.25) -- (0,-0.25) node[midway,below] {$g-d+r$};
\draw[snake=brace,thick] (6.25,4) -- (6.25,0) node[midway,right] {$r+1$};
\end{tikzpicture}
&
$W^r_d(\Gamma) \cong W^\lambda(\Gamma,w)$
\end{array}$\end{center}

The analog of the Brill-Noether number is $g - |\lambda|$. Note in particular that when $\lambda$ is the rectangle above, this is the usual Brill-Noether number $g - (r+1)(g-d+r)$. We call a \emph{marked} metric graph $(\Gamma,w)$ \emph{Brill-Noether general} if for all partitions $\lambda$,  $\dim W^\lambda(\Gamma,w) = g - |\lambda|$ if this number is nonnegative, and $W^\lambda(\Gamma,w)$ is empty otherwise. We obtain the following marked-point version of Theorem \ref{thm_genericity1}.

\begin{thm} \label{thm_genericity2}
Let $(\Gamma,w)$ be a marked chain of cycles with $w$ on the rightmost cycle and torsion profile $\um = (m_2,\cdots,m_g)$. Then $(\Gamma,w)$ is Brill-Noether general if and only if for each $i \in \{2,3,\cdots,g\}$, either $m_i = 0$ or $m_i > i$.
\end{thm}

Our main result is an explicit parameterization of the loci $W^\lambda(\Gamma,w)$, when $\Gamma$ is any chain of cycles of genus $g$ (with arbitrary edge lengths), $w$ is a marked point on the rightmost cycle, and $\lambda$ is any partition. This parameterization is based on combinatorial objects called \emph{$\um$-displacement tableaux} (Definition \ref{def:dt}), which are Young tableaux on $\lambda$ with certain constraints depending on the torsion profile $\um$. We write $t \vdash_{\um} \lambda$ to indicate that $t$ is an $\um$-displacement tableau on the partition $\lambda$. Each such tableau $t$ defines a locus $\TT(t) \subseteq \Pic^0(\Gamma)$, homeomorphic to a torus of dimension equal to $g$ minus the number of symbols appearing in $t$ (Definition \ref{def:Tt}). In particular, $\dim \TT(t) \geq g - |\lambda|$, with equality if and only if $\lambda$ has no repeated symbols.

\begin{thm} \label{t_tori}
For any chain $\Gamma$ of torsion profile $\um$ and any partition $\lambda$, $$W^\lambda(\Gamma,w) = \bigcup_{t \vdash_{\um} \lambda} \TT(t).$$
\end{thm}

A consequence of this (Corollary \ref{c_dp}) is that we can compute the dimensions of the loci $W^\lambda(\Gamma,w)$ by determining the minimum number of distinct symbols in an $\um$-displacement tableau on $\lambda$. Together with semicontinuity results for tropicalization of algebraic curves (summarized in Section \ref{sec_trop}), we can deduce upper bounds on the dimensions of Brill-Noether varieties of (marked or unmarked) algebraic curves. Most importantly, choosing \textit{special} edge lengths in a chain of cycles opens new applications of the theory of linear series on metric graphs, as in our \cite{pfl2}. We also obtain a tropical proof of the generalized Brill-Noether theorem for algebraic curves with one marked point, originally proved by Eisenbud and Harris \cite[Theorem 4.5]{eh}.

\begin{thm}[Generalized Brill-Noether theorem]
\label{thm_gbn}

Let $(C,p)$ be a general marked algebraic curve of genus $g$ over an algebraically closed field of any characteristic. Let $r,d$ be positive integers such that $g-d+r \geq 0$, and let $\alpha = (\alpha_0,\alpha_1,\cdots,\alpha_r)$ be a nondecreasing sequence of nonnegative integers. Then the variety $W^{r,\alpha}_d(C,p) \subseteq \Pic^d(C)$ of line bundles of rank at least $r$ and ramification at least $\alpha$ at the point $p$ (see Definition \ref{def:wrad}) is nonempty if and only if the adjusted Brill-Noether number
$$
\rho(g,d,r,\alpha) = g - (r+1)(g-d+r) - \sum_{i=0}^r \alpha_i
$$
is nonnegative. If nonempty, this locus has dimension exactly $\rho(g,d,r,\alpha)$.
\end{thm}

This theorem will follow from Corollary \ref{cor_gbn}. We now define and give the geometric interpretation of generalized Brill-Noether loci, in both the algebraic and tropical contexts.

\subsection{Brill-Noether loci of marked algebraic curves}  Let $C$ be a smooth projective algebraic curve of genus $g$, and let $d,r$ be nonnegative integers with $r \geq d-g$. Brill-Noether theory concerns the geometry of the schemes $W^r_d(C)$ parameterizing degree $d$ line bundles $L$ on $C$ such that $h^0(L) \geq r+1$. The reason that we require $r \geq d-g$ is that the Riemann-Roch formula guarantees that $h^0(L) \geq (d-g)+1$, regardless of the line bundle. If a marked point $p \in C$ is chosen, one can also impose ramification data at $p$. The \emph{vanishing orders} of a line bundle $L$ of rank $r$ (that is, $h^0(L) = r+1)$ at the point $p$ are the integers $a_0 < a_1 < \cdots < a_r$ such that there exists a section of $L$ vanishing to that order at $p$. Alternatively, one can define $a_i = \max \{ n:\ h^0( L (-n\cdot p) ) \geq r+1-i$. The \emph{ramification orders} are the nondecreasing sequence of integers $\alpha_0 \leq \alpha_1 \leq \cdots \leq \alpha_r$ defined by $\alpha_i = a_i - i$. We therefore make the following definition.

\begin{defn} \label{def:wrad}
Let $(C,p)$ be a smooth projective curve with a marked point. The variety of line bundles of rank at least $r$ with ramification at least $\alpha$ at $p$ is 
$$
W^{r,\alpha}_d(C,p) = \{ L \in \Pic^d(C):\ h^0(L( - (\alpha_i+i) p ) ) \geq r+1 - i\ \mbox{for }i = 0,1,\cdots,r\}.
$$
\end{defn}

To the data $g,r,d,\alpha$, we associate a partition $(\lambda_0,\cdots,\lambda_r)$ as follows (see Figure \ref{fig:wrad_lambda} for a visual illustration).
$$
\lambda_i = (g-d+r) + \alpha_{r-i}
$$

\begin{figure}
\begin{tikzpicture}[scale=0.5]
\draw (0,0) rectangle (6,4);
\foreach \i in {1,2,...,5} \draw (\i,0) -- (\i,4);
\foreach \j in {1,2,3} \draw (0,\j) -- (6,\j);
\draw[snake=brace,thick] (6,-0.25) -- (0,-0.25) node[midway,below] {$g-d+r$};
\draw[snake=brace,thick] (-0.25,0) -- (-0.25,4) node[midway,left] {$r+1$};
\begin{scope}[xshift=0.4cm]
\draw (6,0) rectangle (9,1); \draw (7,0) -- (7,3); \draw (8,0) -- (8,1);
\draw (6,1) rectangle (8,2);
\draw (6,2) rectangle (8,3);
\draw (6,3) rectangle (6,4);
\draw (9,0) node[above right] {$\alpha_r$};
\draw (8,1) node[above right] {$\vdots$};
\draw (8,2) node[above right] {$\alpha_1$};
\draw (6,3) node[above right] {$\alpha_0$};
\end{scope}
\end{tikzpicture}
\caption{The partition $\lambda$ associated to the data $g,d,r,\alpha$. In this example, $r=3,\ d=g-3$, and $\alpha = (0,2,2,3)$.} \label{fig:wrad_lambda}
\end{figure}
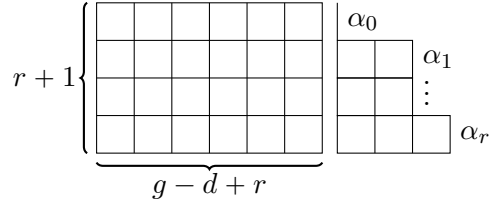

An alternative description that we find useful is that $\lambda$ is the minimal partition whose Young diagram (regarded as a subset of $\ZZ^2_{>0}$) contains all of the points in
$$\{ (g-d+r + \alpha_i, r+1-i):\ i =0,1,\cdots,r\} \cap \ZZ^2_{>0}.$$

Observe that many such varieties $W^{r,\alpha}_d(C,p)$ are automatically isomorphic: if $d$ is increased by $1$, as are all of the integers $\alpha_i$ (that is, if a \emph{base point} is added at $p$), the result is an isomorphic variety, now in $\Pic^{d+1}(C)$, obtained by applying the map $L \mapsto L(p)$ to the original. Therefore there is substantial redundancy in studying all possible data $g,r,d,\alpha$. However, note that any two sets of data $(g,r,d,\alpha)$ that are isomorphic for the reason described give the same partition $\lambda$ (the difference is simply where the gap is drawn in Figure \ref{fig:wrad_lambda}). Indeed, we can identify both with a single subscheme of $\Pic^0(C)$, defined as follows.

\begin{defn} \label{def:wl_alg}
The \emph{Brill-Noether locus} corresponding to the partition $\lambda$ and the marked curve $(C,p)$ is
\begin{align*}
W^\lambda(C,p) =& \{ L \in \Pic^0(C):\ h^0( L(d' \cdot p) ) \geq r'+1 \\&\mbox{ whenever } (g-d'+r',r'+1) \in \lambda\}.
\end{align*}
\end{defn}

\begin{rem}
If $h^0(L(d \cdot p)) \geq r+1$, then it follows automatically that $h^0(L((d+1) \cdot p)) \geq r+1$ and $h^0(L((d-1) \cdot p)) \geq r$. In other words, once $(x,y) \in \lambda$, there are no additional restrictions imposed by $(x-1,y)$ or $(x,y-1)$. This explains why there is no loss of generality in requiring $\lambda$ to be a partition, i.e. to be closed leftward and downward. Also, the Riemann-Roch theorem guarantees that if $r<0$ or $g-d+r \leq 0$, then the condition $h^0(L(d\cdot p)) \geq r+1$ holds for \emph{all} line bundles $L \in \Pic^0(C)$, so there is no loss of generality in only allowing $\lambda$ to have points with both coordinates positive. 
\end{rem}

Comparing Definition \ref{def:wrad} and Definition \ref{def:wl_alg}, it follows that the map $L \mapsto L(-d\cdot p )$ gives an isomorphism $W^{r,\alpha}_d(C,p) \rightarrow W^\lambda(C,p)$, where $\lambda$ is the partition associated to $g,r,d,\alpha$ as described above. Note in particular that when $\lambda$ is a rectangle (i.e. the minimal partition containing $(g-d+r,r+1)$), we have $W^\lambda(C,p) \cong W^r_d(C)$ (regardless of the marked point $p$). Therefore it suffices to study the geometry of $W^\lambda(C,p)$. We hope that the results of this paper will convince the reader that indexing Brill-Noether varieties (with one marked point) by the partition $\lambda$ is the most natural way to study their geometry.

The scheme structure on $W^\lambda(C,p)$ is defined as an intersection of degeneracy loci of certain maps of vector bundles over $\Pic^0(C)$. We summarize this description in Section \ref{sec_ag}, where we also deduce from intersection theory that $\dim W^\lambda(C,p) \geq g - |\lambda|$, with equality if and only if the class of $W^\lambda(C,p)$ (either in the Chow group or in singular cohomology) is

\begin{equation} \label{eq_class}
[W^\lambda(C,p)] = \left. \Theta^{|\lambda|} \middle/ \prod_{(x,y) \in \lambda} \hook(x,y). \right.
\end{equation}

In formula \eqref{eq_class}, $\hook(x,y)$ denotes the hook length of the box $(x,y)$ in the Young diagram of $\lambda$, i.e. the total number of boxes in $\lambda$ that are either of the form $(x',y)$ with $x' \geq x$ or $(x,y')$ with $y' \geq y$. In particular, since $\deg \Theta^g = g!$, this formula shows in the case $g = |\lambda|$ that if $W^\lambda(C,p)$ is $0$-dimensional, then its cardinality is equal, by the hook length formula \cite{frame}, to the number of standard Young tableaux on $\lambda$. Observe in particular that equation \eqref{eq_class} does not depend on $g$ at all, but only on $\lambda$; we offer this as evidence that $\lambda$ is the most natural combinatorial object with which to parameterize Brill-Noether varieties of marked curves.

Theorem \ref{thm_gbn} can therefore be restated to say that, for a general point $(C,p)$ in $\mathcal{M}_{g,1}$ over an algebraically closed field, the dimension of $W^\lambda(C,p)$ is always equal to $g - |\lambda|$ (or $W^\lambda(C,p)$ is empty if $|\lambda| > g$), and formula \eqref{eq_class} gives the class of this locus. This will be proved as Corollary \ref{cor_gbn}.

\subsection{Chains of cycles}
If $\Gamma$ is a metric graph, $w \in \Gamma$ is a point, and $\lambda$ is a partition, then we define a generalized Brill-Noether locus $W^\lambda(\Gamma,w) \subseteq \Pic^0(\Gamma)$ in the same manner as for marked algebraic curves. Here $r(\cdot)$ denotes the Baker-Norine rank of a divisor (see \cite{mz} or \cite{gk}).
\begin{align*}
\wl(\Gamma,w) =& \{ [D] \in \Pic^0(\Gamma):\ r(D + d'\cdot p) \geq r' \\& \mbox{ whenever } (g-d'+r',r'+1) \in \lambda \}
\end{align*}

We will consider a metric graph $\Gamma$ composed of a chain of cycles as shown in Figure \ref{fig_chain}. We require that $v_i$ and $w_i$ are \emph{distinct} points for each $i$. The point $w_g$ will serve as the marked point. We have drawn the chain with bridges from $w_i$ to $v_{i+1}$ to make some of our notation simpler, but all of our results are unaffected if the bridges have length $0$.

\begin{figure} 
\begin{center}
\begin{tikzpicture}
\draw (0,0) circle[radius=1];
\coordinate (w1) at ({cos(70)},{sin(70)});
\coordinate (v1) at ({cos(110)},{sin(110)});

\draw (3,0) circle[radius=1];
\coordinate (w2) at ({3+cos(70)},{sin(70)});
\coordinate (v2) at ({3+cos(110)},{sin(110)});

\draw (7,0) circle[radius=1];
\coordinate (wg) at ({7+cos(70)},{sin(70)});
\coordinate (vg) at ({7+cos(110)},{sin(110)});

\draw[fill=black] (w1) circle[radius=0.05];
\draw[fill=black] (v1) circle[radius=0.05];
\draw[fill=black] (w2) circle[radius=0.05];
\draw[fill=black] (v2) circle[radius=0.05];
\draw[fill=black] (wg) circle[radius=0.05];
\draw[fill=black] (vg) circle[radius=0.05];

\draw[rounded corners = 2ex] (w1) -- (1,1.25) -- (2,1.25) -- (v2);
\draw[rounded corners = 2ex] (w2) -- (4,1.25) -- (4.25,1.25);
\draw[rounded corners = 2ex] (5.75,1.25) -- (6,1.25) -- (vg);
\draw (5,1.25) node {$\cdots$}; 
\draw (5,0) node {$\cdots$};

\draw (v1) node[above] {$v_1$};
\draw (w1) node[above] {$w_1$};
\draw (v2) node[above] {$v_2$};
\draw (w2) node[above] {$w_2$};
\draw (vg) node[above] {$v_g$};
\draw (wg) node[above] {$w_g$};
\end{tikzpicture}
\end{center}
\caption{The chain of cycles $\Gamma$.}\label{fig_chain}
\end{figure}
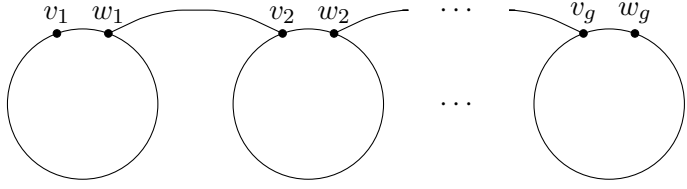

We will show that all of the information about the Brill-Noether loci of $(\Gamma,w_g)$ is encoded in a sequence $\um = (m_2, \cdots, m_g)$ of $g-1$ nonnegative integers, called the \emph{torsion profile} of $\Gamma$. The integer $m_i$ will be called the \emph{$i$th torsion order} and is defined as follows.

\begin{defn} \label{def:tp}
Let $\ell_i$ denote the length of the $i$th cycle, and let $\ell(v_i w_i)$ denote the length of the clockwise edge from $v_i$ to $w_i$. If $\ell(v_i w_i)$ is an irrational multiple of $\ell_i$, then the $i$th torsion order $m_i$ is $0$. Otherwise, $m_i$ is the minimum positive integer such that $m_i \cdot \ell(v_i w_i)$ is an integer multiple of $\ell_i$.
\end{defn}

An equivalent description (which explains the name) is that $m_i$ identifies precisely which multiples of the divisor $w_i - v_i$ are linearly equivalent to $0$. Note that $m_i \neq 1$, since we prohibit $v_i = w_i$.

We summarize this notation in the following situation, for easy reference.

\begin{sit} \label{sit_comb}
Let $\Gamma$ be a chain of cycles with points $v_i,w_i$ as shown in Figure \ref{fig_chain}, and let $m_i$ (for $i=1,2,\cdots,g$) denote the $i$th torsion order of this chain. The \emph{torsion profile} of $\Gamma$ is denoted by $\um = (m_2,\cdots,m_g)$.
\end{sit}

Although we have defined $m_i$ for all $i \in \{1,2,\cdots,g\}$, note that we do not include the first torsion order $m_1$ in the torsion profile $\um$. This is because the value of $m_1$ is determined only by the position of $v_1$ on the first cycle, which in immaterial to the properties of the marked metric graph $(\Gamma,w_g)$.

The Jacobian $\Pic^0(\Gamma)$ is naturally isomorphic to the product of all $g$ cycles of $\Gamma$ (see Lemma \ref{l_stdform} for an explicit isomorphism that we will be the basis of our constructions). Our main result on the structure of $\wl(\Gamma,w_g)$, Theorem \ref{t_tori}, states that $\wlgw$ is always equal to a union of sub-tori, each given by fixing a point on some of the cycles but not others. These tori are in bijection with $\um$-displacement tableaux, which depend on the torsion profile.

The proof of Theorem \ref{t_tori} is largely similar to the method in \cite{cdpr}: special divisors are constructed one cycle at a time, by placing a single chip on each cycle. The choice of where to place the chip of cycle $i$ is determined in \cite{cdpr} by the data of a \emph{lingering lattice path}, which can in turn be specified by a Young tableau on $\lambda$. We give a slightly different construction of divisors from Young tableaux. Unlike in \cite{cdpr}, the divisors we produce are not necessary $v_1$-reduced divisors. However, the dependence of the chip positions on the tableau is more transparent. The idea underlying our construction is that we construct, inductively, divisors on a \emph{sequence of chains} $\Gamma_1,\cdots, \Gamma_g = \Gamma$ of genera $1,2,\cdots, g$, each obtained by taking only the first $i$ cycles of $\Gamma$. Each divisor (when twisted by a multiple of $w_i$ until it has degree $0$) lies in a Brill-Noether locus $W^{\lambda_i}(\Gamma_i,w_i)$, where the partitions $\lambda_i$ form an increasing sequence wherein the box $(x,y)$ is added to the partition $\lambda_i$ if and only if $(x,y)$ has label $i$ in the chosen tableau. Even in the case where we ultimately care only about the loci $W^r_d(\Gamma)$, the more general loci $W^\lambda(\Gamma,w_g)$ are essential, as they are used to record the intermediate steps in this construction. In this way, our method provides a geometric interpretation of the lingering lattice paths of \cite{cdpr}: a specific lingering lattice path records the precise way in which the ramification conditions grow as cycles are attached to $\Gamma$.

\begin{rem}
There are two contexts in algebraic geometry very similar to divisor theory on chains of cycles: limit linear series on chains of elliptic curves and limit linear series (as defined for curves of pseudo-compact type in \cite{oss14dim}) on stable curves formed as a chain of pairs of rational curves joined at two points. In both cases, we suspect that a similar combinatorial analysis should yield a description of the dimension of these spaces of limit linear series, and a procedure to enumerate their irreducible components. Osserman discusses the link between the latter context and that of tropical chains of cycles in \cite{oss14dim}.
\end{rem}

\subsection{Tropicalization}

The interplay between algebraic and tropical geometry will take place in the following situation.

\begin{sit} \label{sit_main}
In addition to the notation of Situation \ref{sit_comb}, let $K$ be a complete valued field with infinite residue field, and let $C$ be a smooth projective curve over $K$. Assume that $C$ has totally split reduction and minimal skeleton isometric to $\Gamma$, and that $w_g$ is rational over the value group of $K$. Let $p \in C$ be a $K$-point specializing to $w_g \in \Gamma$.
\end{sit}

We show in Section \ref{sec_trop} that Theorem \ref{thm_genericity2} implies the following result on algebraic curves.

\begin{thm} \label{thm_introalg}
In Situation \ref{sit_main}, if the torsion profile $\um$ of $\Gamma$ satisfies $m_i = 0$ or $m_i > i$ for all $i$, then $(C,p)$ is Brill-Noether general (as a marked algebraic curve).
\end{thm}

The generalized Brill-Noether Theorem \ref{thm_gbn} will follow, via Corollary \ref{cor_gbn}.

\subsection{Organization of the paper}
We describe displacement tableaux in Section \ref{sec_dt}. The core of the paper is Section \ref{sec_comb}, in which we parameterize the Brill-Noether loci on marked chains of cycles and prove our main result, Theorem \ref{t_tori}, which establishes a correspondence between the components of $\wlgw$ and displacement tableaux. We also deduce Theorems \ref{thm_genericity1} and \ref{thm_genericity2}. Section \ref{sec_ag} summarizes the basic facts about the scheme structure of $W^\lambda(C,p)$ in the algebraic case, and its analysis via intersection theory. Section \ref{sec_trop} proves two semicontinuity properties needed to derive results in algebraic geometry from our tropical results, which we apply to prove Theorem \ref{thm_introalg} and complete our tropical proof of the generalized Brill-Noether Theorem \ref{thm_gbn}. Finally, Section \ref{sec_qu} states some further questions suggested by our work.

\subsection*{Acknowledgements}
The key ideas for this paper came about through conversations with Melody Chan, David Jensen, and Sam Payne. The author is also grateful to Dhruv Ranganathan and Nicola Tarasca for several helpful conversations.

\section{Displacement tableaux} \label{sec_dt}

The description in \cite{cdpr} of special divisors on generic chains gave a bijection between irreducible components of Brill-Noether loci and standard Young tableaux. Our description requires the following generalization of standard tableaux, which incorporates the possibility of special torsion profiles by allowing some symbols from the alphabet to occur more than once.

\begin{defn} \label{def:dt}
Let $\lambda$ be a partition, and let $\um = (m_2,\cdots,m_g)$ be a $(g-1)$-tuple of nonnegative integers. An \emph{$\um$-displacement tableux} on $\lambda$ is a function $t:\ \lambda \rightarrow \{1,2,\cdots,g\}$ satisfying the following properties.
\begin{enumerate}
\item $t$ is strictly increasing in any given row or column of $\lambda$.
\item For any two distinct boxes $(x,y),(x',y')$ in $\lambda$ such that $t(x,y) = t(x',y')$, the following congruence holds. $$x-y \equiv x'-y' \pmod{m_{t(x,y)}}$$
\end{enumerate}
We write $t \vdash_{\um} \lambda$ to denote that $t$ is an $\um$-displacement tableau on $\lambda$.
\end{defn}

Observe that if $t$ satisfies property (1), then it is impossible for two different boxes in $t$ to have the label $1$. Therefore property (2) will never refer to the (unspecified) value of $m_1$.

\begin{eg}
The tableau $$\young(6,3456,1234)$$ is a $(0,3,3,0,2)$-displacement tableau. The only repeated symbols are $3$, $4$, and $6$, each of which occurs twice, in two places separated by a lattice distances $3$, $3$, and $4$, respectively. Since $m_3 \mid 3$, $m_4 \mid 3$, and $m_6 \mid 4$, this is a $(0,3,3,0,2)$-displacement tableau.
\end{eg}

\begin{eg}
Let $\underline{0}$ denote the $(g-1)$-tuple $(0,0,\cdots,0)$. Then a $\underline{0}$-displacement tableau is the same thing as a standard Young tableau with alphabet $\{1, 2, \cdots, g\}$ (note that the alphabet may have more elements than $\lambda$ has boxes, so not all symbols from the alphabet need to appear).
\end{eg}

The following notation will be convenient.

\begin{notn}
Let $\lambda$ be a partition (identified, as usual, with a subset of $\ZZ_{>0} \times \ZZ_{>0}$). Let $\overline{\lambda}$ denote the union of $\lambda$ with all points $(x,y) \in \ZZ \times \ZZ$ such that either $x \leq 0$ or $y \leq 0$.
\end{notn}

\begin{rem}
In some sense, $\overline{\lambda}$ is the more natural set in our application, since the Riemann-Roch formula guarantees that if either coordinate of the point $(g-d+r,r+1)$ is nonpositive, then \emph{all} degree $d$ divisors have rank at least $r$.
\end{rem}

A displacement tableau $t$ may be interpreted as assembly instructions for the partition $\lambda$: boxes are added in $g$ successive steps to produce a sequence of partitions. The permissible boxes to add at any given step are governed by the notion of displacement of a partition, which we now define (see Lemma \ref{lem_disp} for the origin of displacement in our application).

\begin{defn}
Let $\lambda$ be a partition, and let $S$ be a set of integers. The \emph{upward displacement}  of $\lambda$ by $S$ is the partition
$$\disp^+(\lambda,S) = \lambda \cup L,$$
where $L$ is the set of boxes $(x,y) \not\in \lambda$ such that  $x-y \in S$, $(x-1,y) \in \overline{\lambda}$, and $(x,y-1) \in \overline{\lambda}$. The elements of $L$ are called the \emph{loose boxes of $\lambda$ with respect to $S$}.
 \end{defn}

The displacement process is easy to visualize: each element of $z \in S$ causes a slope-$1$ line through $(x,y) = (z,0)$ to be drawn through the Young diagram. Wherever these lines meet an inward corner, they turn that corner outwards, adding a box to the partition. Here, we regard the four corners of the box $(x,y)$ as being the points $(x,y),(x-1,y),(x-1,y-1)$, and $(x,y-1)$ in the plane.

\begin{eg}
Let $\lambda = (7,6,5,1)$, and let $S = 1 + 3 \ZZ$. Then $\disp^+(\lambda,S) = (8,6,5,2)$, as illustrated in the figure below.
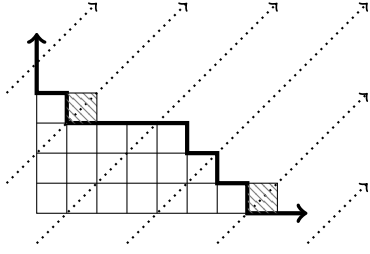
\begin{center}
\begin{tikzpicture}[scale=0.4]
\draw (0,4) -- (0,0) -- (7,0);
\draw (0,4) -- (1,4) -- (1,0);
\draw (0,3) -- (5,3) -- (5,0);
\draw (2,3) -- (2,0);
\draw (3,3) -- (3,0);
\draw (4,3) -- (4,0);
\draw (0,2) -- (6,2) -- (6,0);
\draw (0,1) -- (7,1) -- (7,0);
\draw[style=ultra thick, <->] (0,6) -- (0,4) -- (1,4) -- (1,3) -- (5,3) -- (5,2) -- (6,2) -- (6,1) -- (7,1) -- (7,0) -- (9,0);
\draw[style=dotted,thick,->] (-1,4) -- (2,7);
\draw[style=dotted,thick,->] (-1,1) -- (5,7);
\draw[style=dotted,thick,->] (0,-1) -- (8,7);
\draw[style=dotted,thick,->] (3,-1) -- (11,7);
\draw[style=dotted,thick,->] (6,-1) -- (11,4);
\draw[style=dotted,thick,->] (9,-1) -- (11,1);
\draw[pattern=north west lines,pattern color=gray] (1,3) rectangle (2,4);
\draw[pattern=north west lines,pattern color=gray] (7,0) rectangle (8,1);
\end{tikzpicture}
\end{center}
\end{eg}

\begin{lemma} \label{l_dispmono}
If $\lambda,\lambda'$ are two partitions such that $\lambda \subseteq \lambda'$ and $S$ is any set of integers, then $\disp^+(\lambda,S) \subseteq \disp^+(\lambda',S)$.
\end{lemma}
\begin{proof}
Suppose that $(x,y)$ is a box in $\disp^+(\lambda,S)$ that is not contained in $\lambda'$. Then certainly it is not contained in $\lambda$, so it was loose in $\lambda$. Therefore $(x-1,y)$ and $(x,y-1)$ both lie in $\overline{\lambda}$, and hence in $\overline{\lambda'}$. Thus $(x,y)$ was also loose in $\lambda'$, so it lies in $\disp^+(\lambda',S)$. Therefore every box in $\disp^+(\lambda,S)$ also lies in $\disp^+(\lambda',S)$.
\end{proof}

Displacement of partitions and displacement tableaux are related by the following combinatorial fact.

\begin{cor} \label{c_tabpart}
Let $t$ be an $\um$-displacement tableau on a partition $\lambda$, and let $S_1,S_2,\cdots,S_g$ be sets of integers such that for all boxes $(x,y) \in \lambda$, $x-y \in S_{t(x,y)}$. Define a sequence of partitions $\lambda_0,\lambda_1,\cdots,\lambda_g$ inductively as follows.
\begin{align*}
\lambda'_0 &= \emptyset\\
\lambda'_{i+1} &= \disp^+(\lambda'_i, S_{i+1})
\end{align*}
Then $\lambda \subseteq \lambda'_g$.
\end{cor}

\begin{proof}
Let $\lambda_i = t^{-1}(\{1,2,\cdots,i\})$ be the sub-partition of $\lambda$ given by boxes with labels less than or equal to $i$. By definition of displacement tableaux, all boxes with label $i+1$ are loose in $\lambda_i$ with respect to the set $S_i$. Therefore $\lambda_{i+1} \subseteq \disp^+(\lambda_i,S_{i+1})$. It follows from Lemma \ref{l_dispmono} and induction on $i$ that $\lambda_i \subseteq \lambda'_i$ for all $i$.
\end{proof}

\section{Combinatorics of $\wl(\Gamma,w_g)$} \label{sec_comb}

Let $(\Gamma,w_g)$ be as in Situation \ref{sit_comb}. The language of displacement tableaux allows a succinct and explicit parameterization (Theorem \ref{t_tori}) of all Brill-Noether loci $\wlgw$. We will use the following notation to conveniently denote points of the cycles of $\Gamma$.

\begin{defn}
In Situation \ref{sit_comb}, for any $\xi \in \RR$ let $\la{\xi}_i$ denote the point on the $i$th cycle that is located $\xi \cdot \ell(v_iw_i)$ units clockwise from $w_i$, where $\ell(v_iw_i)$ denotes the length of the clockwise edge from $v_i$ to $w_i$.
\end{defn}

\begin{rem} \label{r_lanot}
In this notation,
\begin{enumerate}
\item The points $v_i$ and $w_i$ are equal to $\la{-1}_i$ and $\la{0}_i$, respectively;
\item If $n_1,n_2$ are \emph{integers}, then $\la{n_1}_i = \la{n_2}_i$ if and only if $n_1\equiv n_2 \pmod{m_i}$, where $m_i$ denotes the torsion order on the $i$th cycle.
\end{enumerate}
\end{rem}

The following lemma gives a convenient bijection between $\Pic^d(\Gamma)$ and the product of the $g$ cycles of $\Gamma$.

\begin{lemma} \label{l_stdform}
Let $D$ be any divisor of degree $d$ on $\Gamma$. Then $D$ is linearly equivalent to a unique divisor of the form $$\sum_{i=1}^g \la{\xi_i}_i + (d-g)\cdot w_g.$$
\end{lemma}

\begin{proof} Assume without loss of generality that $d = g$. So we must show that every degree $g$ divisor is linearly equivalent to a unique sum of $g$ points $\la{\xi_i}_i$, with exactly one point on each cycle.

\textit{Existence of $\xi_i$:} We may assume that the support of $D$ is contained only on the cycles, since any point on a bridge is linearly equivalent to every other point on the the bridge. Next, by adding an integer linear combination of the divisors $w_i - v_{i+1}$ (each of which is linearly equivalent to $0$), we may assume that the portion of $D$ supported on the $i$th cycle has degree $1$. By the Riemann-Roch formula, any degree $1$ divisor on a genus $1$ metric graph is linearly equivalent to an effective degree $1$ divisor. Therefore we can replace the part of $D$ supported on cycle $i$ by a linearly equivalent divisor of the form $\la{\xi_i}_i$ for some $\xi_i$.

\textit{Uniqueness of $\la{\xi_i}_i$:} Suppose that $D = \sum_{i=1}^g \la{\xi_i}_i$ and $D' = \sum_{i=1}^g \la{\xi_i'}_i$ are two linearly equivalent degree $g$ divisors, each consisting of exactly one chip on each cycle. Then $D-D'$ is the principle divisor of a rational function $f$ on $\Gamma$. Since the degree of the part of $D-D'$ to the left of $v_i$ or the right of $w_i$ is $0$, the slope of $f$ is $0$ on all bridges. Therefore, by restricting $f$ to the $i$th cycle, we obtain a rational function on the cycle whose principle divisor is equal to $\la{\xi_i}_i - \la{\xi_i'}_i$. But two points on a cycle are linearly equivalent as divisors if and only if they are the same point: otherwise the rational function whose principle divisor is their difference would have slope $s$ on one interval, and slope $s+1$ on a complementary interval of the cycle, which is impossible since both these slopes would be nonnegative or both nonpositive, with at least one nonzero.
\end{proof}

\begin{rem}
The values $\xi_i$ are easy to compute, given the divisor $D$. Define, for any point $p \in \Gamma$ and index $i$, the number $\widetilde{\xi}_i(p)$ to be $-1$ if $p$ lies to the left of $v_i$, $0$ if $p$ lies to the right of $w_i$, and otherwise to be $\zeta$ such that $p = \la{\zeta}_i$. The function $\widetilde{\xi}_i$ extends to all divisors by linearity. Then the desired values $\xi_i$ are
 $$\xi_i(D) = (i-1) + \widetilde{\xi}_i(D).$$
The value $\xi_i$ is well-defined modulo the length of the the $i$th cycle. Note that $\xi_i(D)$ is not a linear function of the divisor $D$. Instead, it has the feature that $\xi_i(K_\Gamma -D) = -\xi_i(D)$, since $\xi_i(K_{\Gamma}) = 2(i-1)$.
\end{rem}

Lemma \ref{l_stdform} identifies $\Pic^d(\Gamma)$ with the product of the $g$ cycles. Every displacement tableau $t \vdash_{\um} \lambda$ defines a sub-torus in $\Pic^0(\Gamma)$ as follows.

\begin{defn} \label{def:Tt}
Use the notation of Situation \ref{sit_comb}. Let $t$ be an $\um$-displacement tableau on a partition $\lambda$. Denote by $\TT(t)$ the set of divisor classes of the form $$\displaystyle \sum_{i=1}^g \la{\xi_i}_i - g \cdot w_g,$$ where $\xi_1,\cdots,\xi_g$ are real numbers such that $$\xi_{t(x,y)} \equiv x-y \mod{m_{t(x,y)}}$$ for all $(x,y) \in \lambda$.
\end{defn}

Observe that while this definition makes use of the first torsion order $m_1$ of the chain in question, although this number is not included in the torsion profile $\um$. This apparent discrepancy is harmless since the value of $m_1$ is immaterial to the definition of an $\um$-displacement tableau.

\begin{rem}
If $t^\ast$ is the dual tableau of $t$ (obtained by switching the axes), then $\TT(t^\ast)$ is the Serre dual of $\TT(t)$, twisted by $-(2g-2)w_g$ (i.e. the set of classes $[K_\Gamma - D - (2g-2)w_g]$ such that $[D] \in \TT(t)$).
\end{rem}

The definition of $\um$-displacement tableaux ensures that $\TT(t)$ is well-defined and non-empty, since no two boxes of $t$ will impose conflicting conditions on $\xi_i$ (due to the second part of Remark \ref{r_lanot}).

The main result of this paper, Theorem \ref{t_tori} from the introduction, asserts that

 $$W^\lambda(\Gamma,w_g) = \bigcup_{t \vdash_{\um} \lambda} \TT(t).$$

We prove Theorem \ref{t_tori} in Section \ref{ss_pftori}. The key ingredient is Lemma \ref{lem_disp}, which illuminates the link between displacement of partitions and Brill-Noether loci. 

\subsection{Applications and examples of Theorem \ref{t_tori}}
For now, we will assume the result of Theorem \ref{t_tori}, demonstrate how to use it in several examples, and deduce some consequences. 

\begin{cor} \label{c_dp}
The dimension of the largest component of $W^\lambda(\Gamma,w_g)$ is equal to the maximum number of omitted symbols in an $\um$-displacement tableau on $\lambda$. This is equal to $g - |\lambda|$ if and only if it is impossible for an $\um$-displacement tableau on $\lambda$ to have any repeated symbols.
\end{cor}
\begin{proof}
Each $\TT(t)$ is homeomorphic to a torus, whose dimension is equal to the number of symbols $i \in \{1,2,\cdots,g\}$ that do not occur in $t$. This dimension is equal to $g - |\lambda|$ if and only if each symbol in $t$ occurs exactly once.
\end{proof}

\begin{cor} \label{c_gencrit}
The metric graph $\Gamma$ is Brill-Noether general (as a metric graph without a marked point, i.e. in the original sense of \cite{cdpr}) if and only if every $\um$-displacement tableau on a \emph{rectangular} partition has all entries distinct. The marked metric graph $(\Gamma,w_g)$ is Brill-Noether general (as a marked metric graph) if and only if every $\um$-displacement tableau on \emph{any} partition has all entries distinct.
\end{cor}

\begin{eg}
Suppose that $g=4$ and consider the locus $W^1_3(\Gamma)$. This is in bijection with $W^\lambda(\Gamma,w_g)$, where $\lambda = (2,2)$. From the theory of algebraic curves, we expect this to be finite set with two elements, unless $\Gamma$ is special. Indeed, there are two standard Young tableaux on $\lambda$, which correspond to the following two divisor classes of degree $3$.
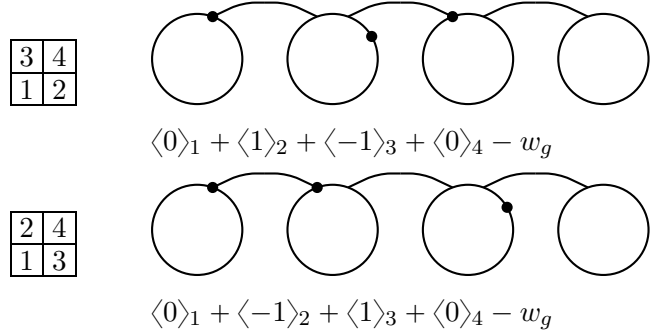
\begin{align*}
\young(34,12) \hspace{1cm} & \begin{tikzpicture}[thick,scale=0.6, every node/.style={scale=0.6}]
\draw (0,0) circle[radius = 1];
\draw[rounded corners = 1ex] ({0+cos(70)},{sin(70)}) -- (1.000000,1.25) -- (1.500000,1.25);
\draw (3,0) circle[radius = 1];
\draw[rounded corners = 1ex] ({3+cos(110)},{sin(110)}) -- (2.000000,1.25) -- (1.500000,1.25);
\draw[rounded corners = 1ex] ({3+cos(70)},{sin(70)}) -- (4.000000,1.25) -- (4.500000,1.25);
\draw (6,0) circle[radius = 1];
\draw[rounded corners = 1ex] ({6+cos(110)},{sin(110)}) -- (5.000000,1.25) -- (4.500000,1.25);
\draw[rounded corners = 1ex] ({6+cos(70)},{sin(70)}) -- (7.000000,1.25) -- (7.500000,1.25);
\draw (9,0) circle[radius = 1];
\draw[rounded corners = 1ex] ({9+cos(110)},{sin(110)}) -- (8.000000,1.25) -- (7.500000,1.25);
\draw[fill=black] ({0+cos(70)},{sin(70)}) circle[radius=0.1];
\draw[fill=black] ({6+cos(110)},{sin(110)}) circle[radius=0.1];
\draw[fill=black] ({3+cos(30)},{sin(30)}) circle[radius=0.1];
\end{tikzpicture}
\\
&\la{0}_1 + \la{1}_2 + \la{-1}_3 + \la{0}_4 - w_g\\
\young(24,13) \hspace{1cm} &\begin{tikzpicture}[thick,scale=0.6, every node/.style={scale=0.6}]
\draw (0,0) circle[radius = 1];
\draw[rounded corners = 1ex] ({0+cos(70)},{sin(70)}) -- (1.000000,1.25) -- (1.500000,1.25);
\draw (3,0) circle[radius = 1];
\draw[rounded corners = 1ex] ({3+cos(110)},{sin(110)}) -- (2.000000,1.25) -- (1.500000,1.25);
\draw[rounded corners = 1ex] ({3+cos(70)},{sin(70)}) -- (4.000000,1.25) -- (4.500000,1.25);
\draw (6,0) circle[radius = 1];
\draw[rounded corners = 1ex] ({6+cos(110)},{sin(110)}) -- (5.000000,1.25) -- (4.500000,1.25);
\draw[rounded corners = 1ex] ({6+cos(70)},{sin(70)}) -- (7.000000,1.25) -- (7.500000,1.25);
\draw (9,0) circle[radius = 1];
\draw[rounded corners = 1ex] ({9+cos(110)},{sin(110)}) -- (8.000000,1.25) -- (7.500000,1.25);
\draw[fill=black] ({0+cos(70)},{sin(70)}) circle[radius=0.1];
\draw[fill=black] ({3+cos(110)},{sin(110)}) circle[radius=0.1];
\draw[fill=black] ({6+cos(30)},{sin(30)}) circle[radius=0.1];
\end{tikzpicture}\\
& \la{0}_1 + \la{-1}_2 + \la{1}_3 + \la{0}_4 - w_g\\
\end{align*}
If $\um = \underline{0}$ (i.e. $\Gamma$ is general), then these are the only points of $W^1_3(\Gamma)$. However, there are two situations in which additional points exist: if $m_2 = 2$ or if $m_3 = 2$. For example, if $m_2 = 2$, there are two additional one-dimensional tori in $W^1_3(\Gamma)$, coming from the two $(0,2,0,0)$-displacement tableaux below. The second loop is drawn differently to show that $m_2 = 2$, and $\Asterisk$s denote either variables or points which are free to vary (giving the degree of freedom of the torus).

\begin{align*}
\young(24,12) \hspace{1cm} & \begin{tikzpicture}[thick,scale=0.600000, every node/.style={scale=0.600000}]
\draw (0,0) circle[radius = 1];
\draw[rounded corners = 1ex] (0.342020,0.939693) -- (0.444626,1.221600) -- (1.700000,0.000000) -- (2.000000,0.000000);
\draw (3,0) circle[radius = 1];
\draw[rounded corners = 1ex] (4.000000,0.000000) -- (4.300000,0.000000) -- (5.555374,1.221600) -- (5.657980,0.939693);
\draw (6,0) circle[radius = 1];
\draw[rounded corners = 1ex] (6.342020,0.939693) -- (6.444626,1.221600) -- (8.555374,1.221600) -- (8.657980,0.939693);
\draw (9,0) circle[radius = 1];
\draw[fill=black] (0.342020,0.939693) circle[radius=0.1];
\draw[fill=black] (2.000000,-0.000000) circle[radius=0.1];
\draw (6.342020,-0.939693) node {$\Asterisk$};
\end{tikzpicture}\\
& \la{0}_1 + \la{1}_2 + \la{\Asterisk}_3 + \la{0}_4 - w_g\\
\\
\young(23,12) \hspace{1cm} & \begin{tikzpicture}[thick,scale=0.600000, every node/.style={scale=0.600000}]
\draw (0,0) circle[radius = 1];
\draw[rounded corners = 1ex] (0.342020,0.939693) -- (0.444626,1.221600) -- (1.700000,0.000000) -- (2.000000,0.000000);
\draw (3,0) circle[radius = 1];
\draw[rounded corners = 1ex] (4.000000,0.000000) -- (4.300000,0.000000) -- (5.555374,1.221600) -- (5.657980,0.939693);
\draw (6,0) circle[radius = 1];
\draw[rounded corners = 1ex] (6.342020,0.939693) -- (6.444626,1.221600) -- (8.555374,1.221600) -- (8.657980,0.939693);
\draw (9,0) circle[radius = 1];
\draw[fill=black] (6.342020,0.939693) circle[radius=0.1];
\draw[fill=black] (0.342020,0.939693) circle[radius=0.1];
\draw[fill=black] (2.000000,-0.000000) circle[radius=0.1];
\draw[fill=red] (9.342020,0.939693) circle[radius=0.1];
\draw (9.342020,0.939693) node[above] {$-1$};
\draw (8.133975,-0.500000) node {$\Asterisk$};
\end{tikzpicture}\\
& \la{0}_1 + \la{1}_2 + \la{0}_3 + \la{\Asterisk}_4 - w_g\\
\end{align*}
\end{eg}

\begin{eg} (\textit{Hyperelliptic chains}) The locus $W^1_2(\Gamma)$ is nonempty if and only if there exists an $\um$-displacement tableau on $\lambda = (g-1,g-1)$. The only possible displacement tableau on this partition is the following.
$$
t = \begin{array}{|c|c|c|c|c|}\hline
2 & 3 & 4 & \cdots & g\\\hline
1 & 2 & 3 & \cdots & g-1\\\hline
\end{array}
$$
This is an $\um$-displacement tableau if and only if $m_2 = m_3 = \cdots = m_{g-1} = 2$. Therefore a chain has a degree $2$ divisor class of rank $1$ if and only if all of these torsion orders are $2$. Theorem \ref{t_tori} also shows that this divisor class is unique: it must be the divisor class of $\la{0}_1 + \sum_{i=2}^g \la{i-3}_i - (g-2) w_g$.
\end{eg}

We close this subsection by showing how Theorem \ref{t_tori} implies Theorems \ref{thm_genericity1} and \ref{thm_genericity2}, which give necessary and sufficient criteria for a chain of cycles to be Brill-Noether general (in the sense of \cite{cdpr}), and for a marked chain of cycles (with marked point on the rightmost cycle) to be Brill-Noether general (as a marked metric graph).

\begin{proof}[Proof of Theorem \ref{thm_genericity1}]
First, we show that the condition that $m_i = 0$ or $m_i > \min(i,g+1-i)$ is \textit{necessary} for Brill-Noether generality. Suppose that $i$ is an index such that $m_i \neq 0$, $m_i \leq g+1-i$, and $m_i \leq i$. Let $\lambda = (m_i,m_i)$ and define a function $t: \lambda \rightarrow \ZZ$ by
\begin{align*}
t(x,1) &= x-m_i + i,\\
t(x,2) &= x+i-1.
\end{align*}
The given inequalities guarantee that $1 \leq t(x,y) \leq g$, and $t$ has exactly one repeated value: $t(m_i,1) = t(1,2) = i$. Therefore $t$ is an $\um$-displacement tableau on a rectangular partition with a repeated entry. It follows from Corollary \ref{c_gencrit} that $\Gamma$ is not Brill-Noether general.

Now we show that the condition is \textit{sufficient}. Suppose that $\Gamma$ is \textit{not} Brill-Noether general; we will show that there exists some $i$ such that $0 < m_i \leq \min(i,g+1-i)$. By corollary \ref{c_gencrit}, there exists a rectangular partition $\lambda$ and a displacement tableau $t \vdash_{\um} \lambda$ with a repeated symbol. Select such a rectangular partition $\lambda$ and tableau $t$ so that $|\lambda|$ is as small as possible. Suppose that $(x,y),(x',y')$ are two distinct boxes containing the index $i$ in $t$. Assume without loss of generality that $x < x'$ and $y > y'$. Then in fact $(x,y)$ is the upper-left corner of $\lambda$, and $(x',y')$ is the lower-right corner, since otherwise one can restrict $t$ to the rectangle with corners $(x,y),(x,y'),(x',y'),(x',y)$ and regard this as an $\um$-displacement tableau on a smaller rectangle, contradicting the minimality of $|\lambda|$. It follows that $x=1,\ y'=1$, and the \emph{only} repetition of symbols in $t$ is $t(1,y) = t(x',1)$. Let this common label be $i$. Then the boxes $(1,y-1),\ (1,y-2),\ \cdots (1,1), (2,1), \cdots, (x'-1,1)$ ($x'+y-3$ boxes total) all have distinct labels strictly less than $i$, so $x' + y - 3 \leq i-1$. Since $t$ is an $\um$-displacement tableau, $(y-1) - (1-x')$ is divisible by $m_i$; since $(y-1)-(1-x') = x' + y - 2$ is positive, it follows that $m_i > 0$ and $m_i \leq x' + y - 2 \leq i$. Similarly, all the boxes $(2,y),(3,y),\cdots,(x',y),(x',y-1),\cdots,(x',1)$ have distinct labels greater than $i$, and it follows that $x' + y  - 3 \leq g-i$, hence $m_i \leq g+1-i$. Therefore $0 < m_i \leq \min(i,g+1-i)$, as desired.
\end{proof}

\begin{proof}[Proof of Theorem \ref{thm_genericity2}]
First, suppose that there exists an index $i$ such that $0 < m_i \leq i$. Let $\lambda$ be the partition $(m_i,1)$, and define a function $t$ in exactly the same way as in the proof of Theorem \ref{thm_genericity1} (but restricted to the smaller, non-rectangular partition). Then $t$ is an $\um$-displacement tableau with the symbol $i$ repeated, so $(\Gamma,w_g)$ is not Brill-Noether general.

Conversely, suppose that $(\Gamma,w_g)$ is not Brill-Noether general. We will show that there exists some $i  \in \{2,3\cdots,g\}$ such that $m_i \leq i$. There exists some partition $\lambda$ with a $\um$-displacement tableau $t$ that repeats some symbol $i$. Suppose that $i = t(x,y) = t(x',y')$, and assume without loss of generality that $x < x'$ and $y > y'$. Then the boxes $(x,y-1), (x,y-2), \cdots, (x,y'), (x+1,y'), \cdots, (x'-1,y')$ constitute $(x-y) - (x'-y') - 1$ boxes with \emph{distinct} labels in $t$. Hence $i \geq (x-y) - (x'-y') > 0$. Since $m_i$ divides $(x-y) - (x'-y')$, it follows that $m_i \leq i$.
\end{proof}

\subsection{Displacement of Weierstrass partitions} 
We now begin the process of proving Theorem \ref{t_tori}. We begin by linking the notion of displacement of partitions to divisors on metric graphs.

\begin{defn}
Let $D$ be a degree $0$ divisor on a marked metric graph $(\Gamma,w)$. Define
$$ \overline{\lambda}_{\Gamma,w}(D) = \{ (g-d+r,\ r+1):\ d,r \in \ZZ,\ r(D + d\cdot w) \geq r \}.$$
If $D$ is a divisor of \emph{any} degree, define $\overline{\lambda}_{\Gamma,w}(D)$ to be $\overline{\lambda}_{\Gamma,w}(D-\deg D \cdot w)$. The \emph{Weierstrass partition} of $D$ at $w$ is the set $$
\lambda_{\Gamma,w}(D) = \overline{\lambda}_{\Gamma,w}(D) \cap \ZZ_{>0}^2.
$$
\end{defn}

\begin{rem}
The word \emph{Weierstrass} is chosen due to the following analogy with \emph{Weierstrass semigroups} on algebraic curve: if $(C,p)$ is a marked algebraic curve, $D$ is a divisor, and $\lambda_{C,p}(D)$ is defined in the same manner as above, then $\lambda_{C,p}(0)$ (the Weierstrass partition of divisor $0$) encodes the data of the Weierstrass semigroup of the point $p$.
\end{rem}

This set $\lambda_{\Gamma,w}(D)$ is (the boxes in the Young diagram of) a partition since $r(D-w) \leq r(D)$ and $r(D+w) \leq r(D) + 1$. The Riemann-Roch formula implies that $\overline{\lambda}_{\Gamma,w}$ always contains all points $(x,y)$ with $x\leq 0$ or $y\leq 0$, so no information is lost in considering only the (finite) set $\lambda_{\Gamma,w}(D)$. The definition of $W^\lambda(\Gamma,w)$ implies that for any partition $\lambda$,
$$[D] \in W^\lambda(\Gamma,w) \mbox{ if and only if } \lambda \subseteq \lambda_{\Gamma,w}(D).$$

The purpose of this subsection and the next is to explain how to compute the Weierstrass partition of any divisor on a chain of cycles as in Situation \ref{sit_comb}. We begin in a more general situation. Let $A$ be any metric graph, and $u \in A$ a point. Let $C$ be a cycle, and $v,w$ be two distinct points on $C$. Denote by $\Gamma$ the metric graph obtained by attaching $C$ to $A$ by an edge connecting $u$ to $v$ (see Figure \ref{fig_disp}). We will explain how the Weierstrass partition of a divisor $D$ at $w$ on $\Gamma$ is related to the Weierstrass partition of the restriction of $D$ to the metric graph $A$ at the marked point $u$. The relationship is surprisingly simple, and it is the origin of our choice of the word displacement: when the marked graph $(A,u)$ is ``displaced'' to the marked graph $(\Gamma,w)$, the Weierstrass partition is displaced upward (in the sense of the $\disp^+$ operation) in a manner than depends on the torsion order of the divisor $w-v$ (which in turn depends only on the edge lengths of $C$).

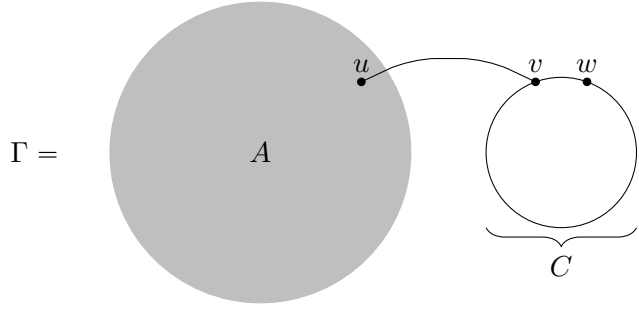
\begin{figure} 
\begin{tikzpicture}
\draw (-3,0) node {$\Gamma = $};

\draw[fill, color = lightgray] (0,0) circle[radius=2cm];
\draw (0,0) node {$A$};
\draw[fill] ({1+cos(70)},{sin(70)}) circle[radius=0.05cm] node[above] {$u$};
\draw[fill] ({4+cos(110)},{sin(110)}) circle[radius=0.05cm] node[above] {$v$};
\draw (4,0) circle[radius=1cm];
\draw[fill] ({4+cos(70)},{sin(70)}) circle[radius=0.05cm] node[above] {$w$};

\draw[rounded corners = 2ex] ({1+cos(70)},{sin(70)}) -- (2,1.25) -- (3,1.25) -- ({4+cos(110)},{sin(110)});

\draw[decorate,decoration={brace,amplitude=0.25cm}] (5,-1) -- (3,-1);
\draw (4,-1.25) node[below] {$C$};
\end{tikzpicture}
\caption{A marked point being displaced across a single cycle.}\label{fig_disp}
\end{figure}

We will denote the points of $C$ as follows, parallel to our notation when describing the chain of cycles. Here $\ell(vw)$ denotes the length of the clockwise edge from $v$ to $w$. $$\la{\xi}= \textrm{the point on $C$ located $\xi\cdot \ell(vw)$ units clockwise from $w$.}$$ Let $m$ denote the nonnegative generator of $\{n \in \ZZ:\ \la{n} = \la{0}\}$, i.e. the torsion order of the divisor $w-v$.

We summarize in the following lemma a few useful facts about metric graphs of this form.

\begin{lemma} \label{l_bridge}
Let $\Gamma$ be the metric graph described above.
\begin{enumerate}
\item Let $D_A,D_A'$ be two divisors on $A$, and let $D_C,D_C'$ be two divisors on $C$. Then $D_A + D_C$ is linearly equivalent to $D_A' + D_C'$ if and only if both
\begin{align*}
D_A + \deg(D_C) \cdot u &\sim D_A' + \deg(D_C')\cdot u \textrm{ as divisors on $A$, and}\\
D_C + \deg(D_A) \cdot v &\sim D_C' + \deg(D_A') \cdot v \textrm{ as divisors on $C$.}
\end{align*}
\item Let $\ell(C)$ denote the length of the cycle $C$. Define a function $\sigma: \Div(C) \rightarrow \textbf{R}/\left( \ell(C)\cdot \ZZ \right)$ by $\la{\xi} \mapsto (\xi+\ell(C) \cdot \ZZ)$, extended by linearity. Two divisors on $C$ are linearly equivalent if and only if they have the same degree and the same image under $\sigma$.
\item Let $D_A,D_C$ be divisors supported on $A$ and $C$ respectively. The linear series $|D_A+D_C|_\Gamma$ is nonempty if and only if there exists an integer $n$ such that both $|D_A+n\cdot u|_A$ and $|D_C - n \cdot v|_C$ are nonempty. Here the subscript in the notation $|\cdot|_X$ indicates the metric graph on which the complete linear series is to be formed.
\end{enumerate}
\end{lemma}

\begin{proof}
The first part follows from the observation that a divisor $D$ is principal on $\Gamma$ if and only if the restriction of $D$ to $A$ differs from a principal divisor by a multiple of $u$ and the restriction of $D$ to $C$ differs from a principal divisor by a multiple of $v$. The second part is standard. The third part is a corollary of the first; we can take $D_A'$ and $D_C'$ to be the restrictions to $A$ and $C$ of any element of the complete linear series of $D_A + D_C$, after first moving all chips on the edge $uv$ to one end or the other.
\end{proof}

The following lemma is the key to all of our analysis of chains of cycles.

\begin{lemma} \label{lem_disp}
Let $D$ be any divisor of degree $g-1$ on $A$, where $g$ is the genus of $\Gamma$, and $p$ any point on $C$. Then
$$\lambda_{\Gamma,w}(D+p) = 
\begin{cases}
\disp^+(\lambda_{A,u}(D), z+m\ZZ) & \mbox{if $p=\la{z}$ for some integer $z$,}\\
\lambda_{A,u}(D) & \mbox{otherwise.}
\end{cases}$$
\end{lemma}

\begin{proof}
Throughout this proof, we will use the following vocabulary to make several statements more succinct: we call a divisor $D$ on a marked metric graph $(\Gamma,w)$ a \emph{weak $g^r_d$ for $(\Gamma,w)$} if $r(D + (d-\deg D) w) \geq r$. In other words, a weak $g^r_d$ differs from a $g^r_d$ by a multiple of the marked point\footnote{In the context of algebraic curves, ``a $g^r_d$'' usually means a \emph{linear series}, i.e. a choice of both a degree $d$ divisor class and a $r+1$-dimensional space of sections of the associated line bundle. In our context, there is nothing analogous to a vector space of sections, so we will use the phrase to refer only to the divisor class.}.

Fix positive integers $x,y$, and define integers $r = y-1$ and $d = y-x+g-1$. Then $(x,y) \in \lambda_{\Gamma,w}(D+p)$ if and only if $D+p$ is a weak $g^r_d$ for $(\Gamma,w)$. We will establish the following two claims.
\begin{enumerate}
\item If $p \neq \la{x-y}$, then $D+p$ is a weak $g^r_d$ for $\Gamma$ if and only if $D$ is a weak $g^r_{d-1}$ for $A$.
\item If $p = \la{x-y}$, then $D+p$ is a weak $g^r_d$ for $\Gamma$ if and only if $D$ is both a weak $g^r_d$ and a weak $g^{r-1}_{d-2}$ for $A$.
\end{enumerate}
Translated into the language of Weierstrass partitions, these claims assert that $(x,y) \in \lambda_{\Gamma,w}(D+p)$ if and only if either:
\begin{enumerate}
\item[(1')] $p \neq \la{x-y}$ and $(x,y) \in \lambda_{A,u}(D)$, or
\item[(2')] $p = \la{x-y}$ and both $(x,y-1)$ and $(x-1,y)$ lie in $\overline{\lambda}_{A,u}(D)$.
\end{enumerate}

We now show that these two claims imply the lemma. Consider first the case where $p = \la{z}$ for some integer $z$. Then the set of all integers $z'$ such that $\la{z'} = \la{z}$ is the arithmetic progression $z + m \ZZ$, hence $p = \la{x-y}$ if any only if $x-y \equiv  z \pmod{m}$. Therefore, claims (1) and (2) amount, in this case, to saying that $(x,y) \in \lambda_{\Gamma,w}(D+p)$ if and only if $(x,y) \in \disp^+(\lambda_{A,u}(D),z+m \ZZ)$. Consider next the case where $p$ is not equal to $\la{z}$ for any integer $z$. Then $p \neq \la{x-y}$ for all choices of $(x,y)$, so claim (1) shows that $\lambda_{\Gamma,w}(D+p) = \lambda_{A,u}(D)$ in this case. Therefore the two claims will imply the lemma.

\textit{Proof of claims (1) and (2)}. Let $\xi$ be a real number such that $p = \la{\xi}$. Observe that the divisor $D+p+(d-g)w$ is linearly equivalent to $D + (d-g)v + \la{\xi+d-g}$ by Lemma \ref{l_bridge}, which is equal to $D + (d-g)v + \la{\xi+y-x-1}$. Let $D'$ denote this degree $d$ divisor.

By definition, $D+p$ is a weak $g^r_d$ for $(\Gamma,w)$ if and only if $|D' - E| \neq \emptyset$ for all effective divisors $E$ of degree $r$. Replacing $E$ by a linearly equivalent divisor if necessary, it suffices to consider only effective divisors $E$ with at most one chip on $C$, and no chips on the interior of the edge from $u$ to $v$.

\textit{Case 1: $E$ has one chip on $C$ and $r-1$ chips on $A$.} Let the chip on $C$ be placed at the point $\la{\delta}$, and let $E_A = E-\la{\delta}$ be the part supported on $A$. By Lemma \ref{l_bridge}, $|D'-E| \neq \emptyset$ if and only if there exists an integer $n$ such that $|D - E_A + n\cdot u|_A \neq \emptyset$ and $|(d-g-n)\cdot v + \la{\xi+y-x-1}-\la{\delta}|_C \neq \emptyset$. By Lemma \ref{l_bridge} and the fact that $v = \la{-1}$, the second condition is equivalent to $$| (d-g-n+1)\cdot v - \la{\delta - \xi -y + x}| \neq \emptyset \textrm{ on $C$.}$$

	This holds either if $d-g-n+1 \geq 2$ or if $d-g-n+1 = 1$ and $\la{\delta - \xi - y + x} = v$. Therefore, as long as $\la{\delta}$ is not equal to one specific point (namely, $\la{\xi+y-x-1}$), the maximum value of $n$ satisfying the second condition of Lemma \ref{l_bridge}(3) is $n = d-g-1$. Hence assuming that $\la{\delta}$ is not equal to this point, $|D'-E| \neq \emptyset$ on $\Gamma$ if and only if $|D + (d-g-1)u - E_A| \neq \emptyset$. This holds for \emph{every} choice of degree-$(r-1)$ effective $E_A$ on $A$ if and only if $D$ is a weak $g^{r-1}_{d-2}$ for $A$. 

We conclude that $|D+p-E| \neq \emptyset$ for \emph{all} degree $r$ effective $E$ with exactly one chip on $C$ if and only if $D$ is a weak $g^{r-1}_{d-2}$ on $A$.

\textit{Case 2: $E$ is supported on $A$. }Again, we analyze the maximum value of $n$ such that when $n \cdot v$ is subtracted from the part of $D'-E$ supported on $C$, the result has nonempty linear series. The part of $D'-E$ supported on $C$ consists of $(d-g)v + \la{\xi+y-x-1}$, hence there are two cases: if $\la{\xi+y-x-1} = v$, then the maximum $n$ is $d-g+1$, and otherwise it is $d-g$. The hypothesis of the first case is equivalent to $p = \la{x-y}$. Define $$\epsilon = \begin{cases} 1 & \mbox{if } p = \la{x-y}\\ 0 & \mbox{otherwise.}\end{cases}$$

Lemma \ref{l_bridge} now says that $|D'-E| \neq \emptyset$ on $\Gamma$ if and only if $|D+(d~-~g~+~\epsilon)u - E| \neq \emptyset$ on $A$. This holds for every effective degree $r$ divisor on $A$ if and only if $D$ is a weak $g^r_{d-1+\epsilon}$ for $A$.

\textit{End of the proof:} combining these two cases, we see that $D+p$ is a weak $g^r_d$ for $\Gamma$ if and only if $D$ is both a weak $g^{r-1}_{d-2}$ and a weak $g^r_{d-1+\epsilon}$ for $A$. If $p = \la{x-y}$, i.e. $\epsilon = 1$, this is precisely claim (2). If $p \neq \la{x-y}$, i.e. $\epsilon = 0$, then this is equivalent to saying only that $D$ is a weak $g^r_{d-1}$, since a weak $g^r_{d-1}$ is automatically also a weak $g^{r-1}_{d-2}$; this gives claim (1) and completes the proof of the lemma.
\end{proof}

\subsection{Proof of Theorem \ref{t_tori}} \label{ss_pftori}

We now return Situation \ref{sit_comb}, where $\Gamma$ is a chain of cycles with torsion profile $\um$. By induction on the genus (where the base case may be taken to be a genus $0$ metric graph consisting of a single  vertex), Lemma \ref{lem_disp} immediately gives the a description, in terms of displacement, of the Weierstrass partition of a divisor of the form $D = \sum_{i=1}^g \la{\xi_i}_i$, as follows.

\begin{cor} \label{c_computewp}
Let $D$ be the degree $g$ divisor $\sum_{i=1}^g \la{\xi}_i$ on the marked graph $(\Gamma,w_g)$ of Situation \ref{sit_comb}. For $i=1,2,\cdots,g$, let $S_i = \{z \in \ZZ:\ \la{\xi_i}_i = \la{z}_i \}$. This is either empty, a set with one element (in which case $m_i =0$), or a congruence class modulo $m_i$. Define partitions $\lambda_0, \lambda_1,\cdots,\lambda_g$ as follows.
\begin{align*}
\lambda_0 &= \emptyset\\
\lambda_{i+1} &= \disp^+(\lambda_i,S_{i+1})
\end{align*}
Then the Weierstrass partition of $D$ on $(\Gamma,w_g$) is $\lambda_g$.
\end{cor}

\begin{lemma} \label{l_tori1}
If $t \vdash_{\um} \lambda$, then $\TT(t) \subseteq \wlgw$.
\end{lemma}
\begin{proof}
Suppose that $[D] \in \TT(t)$. Define sets $S_1,S_2, \cdots,S_g$ and partitions $\lambda_0,\lambda_1,\cdots,\lambda_g$ as in corollary \ref{c_computewp}; by that corollary, $\lambda_g$ is the Weierstrass partition of $D$ for $(\Gamma,w_g)$. It suffices to prove that $\lambda \subseteq \lambda_g$. By the definition of $\TT(t)$, the sets $S_i$ satisfy the condition that $x-y \in S_{t(x,y)}$ for all $(x,y) \in \lambda$. Therefore $S_i$ also meet the hypothesis of Corollary \ref{c_tabpart}, from which it follows that $\lambda \subseteq \lambda_g$, as desired.
\end{proof}

\begin{lemma} \label{l_tori2}
If $[D] \in \wlgw$, then there exists an $\um$-displacement tableau $t$ such that $[D] \in \TT(t)$.
\end{lemma}
\begin{proof}
By Lemma \ref{l_stdform}, we may assume without loss of generality that $D+g\cdot w_g$ consists of exactly one chip on each cycle of $\Gamma$. Let the position of the chip on the $i$th cycle be $\la{\xi_i}_i$.

For each $i \in \{1,2,\cdots,g\}$, denote by $\Gamma_i$ the first $i$ cycles of the chain $\Gamma$ (with the bridges between them), and denote by $D_i$ the restriction of $D+g\cdot w_g$ to $\Gamma_i$. Let $\lambda_i$ be the Weierstrass partition of $D_i$ for the marked metric graph $(\Gamma_i,w_i)$. Then by Lemma \ref{lem_disp}, these partitions are related by
$$\lambda_{i+1} = \disp^+(\lambda_i,S_{i+1}),$$
where $S_i = \{z \in \ZZ:\ \la{\xi_i}_i = \la{z}_i \}$. In particular, all of the boxes $(x,y)$ in $\lambda_{i+1} \backslash \lambda_i$ have values of $x-y$ that are congruent modulo $m_i$.

Define a function $t:\ \lambda \rightarrow \{1,2,\cdots,g\}$ by 
$$t(x,y) = \min \{i:\ (x,y) \in \lambda_i \}.$$
The previous paragraph shows that $t$ is in fact an $\um$-displacement tableau. Each $(x,y) \in \lambda$ is a loose box of $\lambda_{t(x,y)-1}$ with respect to the set $S_{t(x,y)}$, so in particular $x-y \in S_{t(x,y)}$, i.e. $\la{\xi_{t(x,y)}}_{t(x,y)} = \la{x-y}_{t(x,y)}$. Therefore $[D] \in \TT(t)$.
\end{proof}

Theorem \ref{t_tori} follows directly from Lemmas \ref{l_tori1} and \ref{l_tori2}.

\section{Algebraic geometry of $\wl(C,p)$} \label{sec_ag}

Throughout this section, let $C$ denote a smooth projective curve over a field $K$, and $p \in C(K)$ a $K$-point. We will define the structure of $W^\lambda(C,p)$ as a subscheme of $\Pic^0(C)$, and compute its expected dimension and class via intersection theory.

We first review the scheme structure of $W^r_d(C)$, following \cite[IV \S 3]{acgh}. If an effective divisor $E$ of degree at least $2g-d-1$ is fixed, then the vector spaces $H^0( L(E) )$ (where $L$ varies over all degree $d$ line bundles) all have the same dimension $\deg E + d-g + 1$, and form the fibers of a vector bundle $\mathcal{M}$ over $\Pic^d(C)$. A second vector bundle $\mathcal{P}$ may be formed, whose fiber over $[L] \in \Pic^d(C)$ is $H^0( L(E) / L)$. There is an obvious map $\mathcal{M} \rightarrow \mathcal{P}$ of vector bundles. Observe that for any particular line bundle $L$ of degree $d$, $h^0(L) \geq r+1$ if and only if $\dim \ker \left( H^0(L(E)) \rightarrow H^0(L(E)/L)\right) \geq (r+1)$. Therefore $W^r_d(C)$ is defined to be the subscheme of $\Pic^d(C)$ whose defining equations are, locally, the $N \times N$ minors of a trivialization of the bundle map $\mathcal{M} \rightarrow \mathcal{P}$, where $N = (\deg E + d - g - r+1)$. The details of the construction of the vector bundles $\mathcal{M},\mathcal{P}$, as well as the fact that this scheme does not depend on any of the choices made, can be found in \cite{acgh}.

Each scheme $W^r_d(C)$ may be regarded as a subscheme of $\Pic^0(C)$, via the isomorphism $L \mapsto L(-d \cdot p)$.

\begin{defn} \label{d_wlfirst}
The scheme $W^\lambda(C,p)$ is the intersection of all the schemes $W^r_d(C)$ such that $(g-d+r,r+1) \in \lambda$, where these are regarded as subschemes of $\Pic^0(C)$.
\end{defn}

Definition \ref{d_wlfirst} is the most natural definition, but it is difficult to use since the various schemes being intersected are not transverse. We therefore give an alternative characterization, to which we can apply intersection theory.

First, since we wish to describe a locus in $\Pic^0(C)$ rather than $\Pic^d(C)$, we may replace $L$ by $L(d \cdot p)$ in the discussion above.

Next, suppose that in the construction of the schemes $W^r_d(C)$ we choose the divisor $E$ to be $(2g-1-d) \cdot p$. Then the same vector bundle $\mathcal{M}$ can be used in the construction of every locus $W^r_d(C)$; its fiber over $[L]$ will be $H^0(L((2g-1)p))$. The bundle $\mathcal{P}$ used in the definition of $W^r_d(C)$ has fibers naturally identified with $H^0( L((2g-1) \cdot p) / L(d \cdot p))$ (note that we now write $L(d \cdot p)$ since $L$ is a degree $0$ line bundle, and $L(d \cdot p)$ is the line bundle which is required to have $r+1$ sections). Write this bundle as $\mathcal{P}_d$. These vector bundles, for various choices of $d$, fit into a flag of quotients.

$$
\mathcal{P} = \mathcal{P}_{-1} \twoheadrightarrow \mathcal{P}_{0}
\twoheadrightarrow \mathcal{P}_{1}
\twoheadrightarrow \cdots
\twoheadrightarrow \mathcal{P}_{2g-1} = 0
$$

In this notation, we can alternatively characterize $W^\lambda(C,p)$ as the following degeneracy locus.

\begin{align*}
W^\lambda(C,p) =& \left\{[L] \in \Pic^0(C):\ \textrm{rank}( \mathcal{M}_{[L]} \rightarrow \left(\mathcal{P}_d)_{[L]} \right) \leq g-r-1\right. \\
& \left.\mbox{ whenever } (r+1,g-d+r) \in \lambda \right\}
\end{align*}

Degeneracy loci of this form are analyzed in \cite{fulton}. The expected codimension of this degeneracy locus is equal to $|\lambda|$, and the class of the locus, in case it has the correct dimension, can be computed from the Chern classes of $\mathcal{M}$ and $\mathcal{P}_d$. By \cite[VII \S 4]{acgh} (where the same bundles are used to analyze $W^r_d(C)$), $\mathcal{P}_i$ has trivial Chern classes, and $\mathcal{M}$ has Chern classes given by 
$$
c_i(-\mathcal{M}) = \Theta^i / i!.
$$

By \cite[Theorem 10.1]{fulton}, this degeneracy locus supports a (Chow or singular cohomology) class given by the following determinant, with this class being equal to the class of the degeneracy locus in case the dimensions match. Here $\lambda^*_1$ denotes the largest element of the dual partition, i.e. the number of nonzero elements of $\lambda$.

\begin{align*}
\Omega^\lambda(C) &= \det \left( c_{\lambda_i-i+j} ( - \mathcal{M} ) \right)_{1 \leq i,j \leq \lambda^*_1}\\
&= \det \left( \Theta^{\lambda_i - i + j} / (\lambda_i-i+j)! \right)_{1 \leq i,j \leq \lambda^*_1}\\
&= \Theta^{|\lambda|} \cdot \det \left( 1 / (\lambda_i - i + j)! \right)_{1 \leq i,j \leq \lambda^*_1}
\end{align*}

In the last line, the number $1 / n!$ should be interpreted to be $0$ when $n < 0$. By Aitken's determinantal formula \cite{aitken}, this determinant is equal to the number of standard Young tableaux on $\lambda$ divided by $|\lambda|!$; by the hook length formula \cite{frame} we can express this as the product of the reciprocals of the hook lengths of $\lambda$.

We therefore deduce the following fact.

\begin{prop} \label{prop_lowerbound}
If $g \geq |\lambda|$, then the Brill-Noether locus $W^\lambda(C,p)$ is nonempty of dimension at least $g - |\lambda|$. If the dimension of $W^\lambda(C,p)$ is exactly $g - |\lambda|$, then its (Chow or singular cohomology) class is equal to
$$
\left. \Theta^{|\lambda|} \middle/ \prod_{(x,y) \in \lambda} \hook(x,y) \right. .
$$
\end{prop}

\section{Tropicalization} \label{sec_trop}

This section is concerned with the marked curve $(C,p)$ from Situation \ref{sit_main}. We will prove Theorem \ref{thm_introalg}, which is a corollary of a pair of ``semicontinuity'' results.

The retraction from the Berkovich analytic space $C^{\textrm{an}}$ to its skeleton extends by linearity to divisors and induces a map

$$
\Trop:\ \Pic(C) \rightarrow \Pic(\Gamma).
$$

When $K$ is the field of fractions of a discrete valuation ring, this map is identical to the tropicalization map of \cite{baker08}. This map has two important properties, which can be interpreted as two forms of semicontinuity for $\Trop$.

\begin{prop}\ \label{prop_semi}
\begin{enumerate}
\item The image of $W^\lambda(C,p)$ under $\Trop$ lies in the locus $W^\lambda(\Gamma,w_g)$.
\item $\dim \wlgw \geq \dim W^\lambda(C,p)$.
\end{enumerate}
\end{prop}

\begin{proof}
The definition of $\Trop$ implies that for any divisor $D$ on $C$ and any integer $n$, $\Trop(D + n \cdot p)  = \Trop(D) + n w_g$. Therefore it suffices to show that $\Trop(W^r_d(C)) \subseteq W^r_d(\Gamma)$ for all integers $r,d$. This is precisely the specialization inequality of \cite{baker08}; although that paper works in the discrete valuation case, the proof applies without modification to the general case. This establishes part (1).

The tropical Jacobian $\Pic^0(\Gamma)$ is isomorphic to the skeleton $\Sigma(\Pic^0(C)^{\textrm{an}})$ \cite{br}, and this isomorphism is compatible with the tropicalization map. Therefore Gubler's results on tropicalization of subvarieties of abelian varieties apply: by \cite[Theorem 6.9]{gubler}, a variety of pure dimension $d$ in $\Pic^0(C)$ has image, under the retraction to the skeleton $\Sigma(\Pic^0(C))$, of pure dimension $d$. Combining this with part (1), it follows that the image $\Trop(W^\lambda(C,p)) \subseteq \Gamma$ contains a tropical variety of dimension equal to the dimension of $W^\lambda(C,p)$, and part (2) follows.
\end{proof}

\begin{cor}[Theorem \ref{thm_introalg}]
In Situation \ref{sit_main}, if for each $i$ either $m_i = 0$ or $m_i > i$, then for every partition $\lambda$, $W^\lambda(C,p)$ has the expected dimension.
\end{cor}
\begin{proof}
Theorem \ref{thm_genericity2} and Proposition \ref{prop_semi}.
\end{proof}

\begin{cor} \label{cor_gbn}
Working over any algebraically closed field, if $(C,p)$ is a general marked curve, then for every partition $\lambda$, $W^\lambda(C,p)$ has the expected dimension, and Chow class given by Formula \eqref{eq_class}.
\end{cor}
\begin{proof}
Theorem \ref{thm_introalg} and deformation theory (cf. \cite[Appendix]{baker08}) show the existence of Brill-Noether general marked curves over any complete valued field with infinite residue field. The existence of Brill-Noether general marked curves over an arbitrary algebraically closed field follows from the fact that the coarse moduli space $M_{g,1}$ of marked smooth curves is a scheme of finite type over $\textrm{Spec} \ZZ$, and the locus of Brill-Noether general points is Zariski open and surjects onto $\textrm{Spec} \ZZ$.

Proposition \ref{prop_lowerbound} shows that the Chow class is given by formula \eqref{eq_class}.
\end{proof}

The generalized Brill-Noether theorem (Theorem \ref{thm_gbn}) follows from Corollary \ref{cor_gbn}.

\section{Further questions for arbitrary metric graphs} \label{sec_qu}

The detailed description, for marked chains of cycles $(\Gamma,w_g)$, of the Brill-Noether loci $W^\lambda(\Gamma,w_g)$, makes it possible to see an intriguing analogy with the theory of algebraic curves; we believe that extending this analogy to arbitrary metric graphs might yield further insight. In this section we formulate this analogy and the questions arising from it.

For a marked algebraic curve $(C,p)$, classical intersection theory, as discussed in Section \ref{sec_ag}, guarantees not only that $W^{\lambda}(C,p)$ is nonempty (when $g-|\lambda| \geq 0$), but that all of its irreducible components are at least $(g-|\lambda|)$-dimensional and that it supports a particular intersection class, namely $\Theta_C^{|\lambda|} / \prod_{(x,y) \in \lambda} \hook(x,y)$. Furthermore, this is precisely the class of $W^\lambda(C,p)$ when its dimension is exactly $g - |\lambda|$.

Even though there is currently no equivalent in tropical intersection theory of the machinery available to study degeneracy loci in algebraic geometry (as in \cite{fulton}), the same facts outlined in the previous paragraph hold in the context of chains of cycles, when interpreted in a suitable way. Indeed, observe that if $(\Gamma,w_g)$ is a marked chain of cycles as in Situation \ref{sit_comb}, and $\lambda$ is a partition, then regardless of the torsion profile there is always a locus given by those tableaux that have no repeating symbols, namely
$$
\bigcup_{t \vdash_{\underline{0}} \lambda} \TT(t) \subseteq W^\lambda(\Gamma,w_g).
$$
The parameterization provided by Defintion \ref{def:Tt} shows that this locus varies continuously as the edge lengths of $\Gamma$ vary. In this section we will refer to this subset as the \emph{stable locus}. It consists of a union of $(g-|\lambda|)$-dimensional tori. Furthermore, observe that any point of $W^\lambda(\Gamma,w_g)$ that is not in the stable locus necessarily sits in a torus $\TT(t)$ coming from a tableau $t$ with a repeated symbol, i.e. a torus of dimension strictly greater than $g-|\lambda|$. Thus the locus $W^\lambda(\Gamma,w_g)$ is equal to the stable locus if and only if it is exactly $(g-|\lambda|)$-dimensional.

To complete the analogy with algebraic curves, we point out that, in a suitable sense, the stable locus has intersection class $\Theta_\Gamma^{|\lambda|} /  \prod_{(x,y) \in \lambda} \hook(x,y)$. We must define what we mean by these terms. 

For any metric graph $\Gamma$ of genus $g$, there is a tropical subvariety $\Theta_\Gamma \subset \Pic^{g-1}$, called the tropical theta divisor, which can be identified with $W^{0}_{g-1}(\Gamma)$; see \cite{mz}. As usual, we suppose that we have chosen a marked point $w \in \Gamma$, and regard $\Theta_\Gamma$ as a subset of $\Pic^0(\Gamma)$. When $\Gamma$ is a chain of cycles, Theorem \ref{t_tori} shows that $\Theta_\Gamma$ is a union of $(g-1)$-dimensional tori (the Young diagram associated to $r=0,\ d=g-1$ is a single box), each obtained by fixing one coordinate $\xi_i$ to be equal to $0$ (in the notation of Lemma \ref{l_stdform}). If $d$ is a positive integer, and we choose $d$ general translates of $\Theta_\Gamma$, we can easily describe the intersection of these translates: it is a union of $g (g-1) \cdots (g-d+1)$ sub-tori, each giving by fixing $d$ coordinates $\xi_i$. More precisely, for each choice of $d$ coordinates, the intersection includes $d!$ tori given by fixing these $d$ coordinates.

Now, observe that the stable locus of $W^\lambda(\Gamma,w_g)$ has a similar description: for each standard Young tableaux on $\lambda$ and choice of $|\lambda|$ symbols from the set $\{1,2,\cdots, g\}$, there is one torus in the stable locus, given by fixing the $|\lambda|$ coordinates chosen. By the hook-length formula \cite{frame}, the number of tori in the stable locus fixing a given set of $|\lambda|$ coordinates is
$$
|\lambda| ! / \prod_{(x,y) \in \lambda} \hook(x,y).
$$
Provided that any subtorus of $\Pic^0(\Gamma)$ is regarded as ``equivalent'' to any translate of itself, it follows that the stable locus of $W^\lambda(\Gamma,w_g)$ may be regarded as being equivalent to $1 / \prod_{(x,y) \in \lambda} \hook(x,y)$ times the locus given by intersecting $|\lambda|$ general translates of $\Theta_\Gamma$. This precisely mirrors the formula for the expected class of $W^\lambda(C,p)$ in the algebraic case.

To formulate a more general statement for arbitrary marked metric graphs $(\Gamma,w)$, we require a suitable equivalence relation on tropical subvarieties of $\Pic^0(\Gamma)$. One candidate is the notion of rational equivalence described in \cite{ahr}. Whatever choice is made, it should be compatible with tropical intersection theory, e.g. as developed in \cite{ar}.

To extend the analogy outlined above to arbitrary marked metric graphs, several questions present themselves.

\begin{qu}
If $(\Gamma,w)$ is any marked metric graph, is the local dimension of $W^\lambda(\Gamma,w)$ at least $g- |\lambda|$ at every point?
\end{qu}

Note that the semicontinuity results of Section \ref{sec_trop} imply that the local dimension of $W^\lambda(\Gamma,w)$ is at least $g-|\lambda|$ at every point that lifts to a point of $W^\lambda(C,p)$ on an algebraic marked curve. This does not rule out the existence of extra components of smaller dimension.

\begin{qu}
If $(\Gamma,w)$ is any marked metric graph, does the locus $W^\lambda(\Gamma,w)$ always support a tropical subvariety equivalent (in a suitable sense) to $\Theta_\Gamma^{|\lambda|} / \prod_{(x,y) \in \lambda} \hook(x,y)$? Is there a ``canonical'' such locus that varies continuously as the edge lengths vary (analogous to the ``stable'' locus defined above for chains of cycles)?
\end{qu}

\begin{qu}
If $\Gamma$ is a metric graph, $w$ is a marked point, and $\lambda$ is a partition such that $\wl(\Gamma,w)$ has pure dimension $g - |\lambda|$, must $W^\lambda(\Gamma,w)$ be equivalent (in a suitable sense) to $\Theta_\Gamma^{|\lambda|} / \prod_{(x,y) \in \lambda} \hook(x,y)$?
\end{qu}

\bibliography{main}{}
\bibliographystyle{alpha}

\end{document}